\documentclass[12pt]{article}
\usepackage{latexsym, amssymb, amsmath, amscd, amsfonts, epsfig, graphicx, colordvi,verbatim,ifpdf,extarrows}
\usepackage{amsfonts, amsmath, amssymb}
\usepackage{amssymb,amsfonts,amsmath,latexsym,epsfig,cite, psfrag,eepic,color}
\usepackage{amscd,graphics}
\usepackage{latexsym, amssymb,  amsmath,amscd, amsfonts, epsfig, graphicx, colordvi,amsthm}

\usepackage{graphicx}
\usepackage{epstopdf}
\usepackage{color}
\usepackage{ifpdf}
\usepackage{fancybox}
\usepackage[font=small,labelfont=bf,labelsep=none]{caption}
\usepackage{float}

\allowdisplaybreaks

\usepackage[latin1]{inputenc}

\newtheorem{theorem}{Theorem}[section]

\theoremstyle{remark}

\theoremstyle{definition}
\newtheorem{definition}[theorem]{Definition}
\newtheorem{lemma}[theorem]{Lemma}

\def\qed{\nopagebreak\hfill{\rule{4pt}{7pt}}
\medbreak}

\numberwithin{equation}{section}

\def\qed{\nopagebreak\hfill{\rule{4pt}{7pt}}
\medbreak}

\newlength{\boxedparwidth}
  {\begin{center} \begin{tabular}{|@{\hspace{.315in}}c@{\hspace{.15in}}|}
                  \hline \\ \begin{minipage}[t]{\boxedparwidth}
                  \setlength{\parindent}{.25in}}%
  {\end{minipage} \\ \\ \hline \end{tabular} \end{center}}

\begin{document}

\begin{center}
{\Large \bf Separable integer partition classes and partitions with congruence conditions}
\end{center}

\begin{center}
 {Thomas Y. He}$^{1}$, {C.S. Huang}$^{2}$,
  {H.X. Li}$^{3}$ and {X. Zhang}$^{4}$ \vskip 2mm

$^{1,2,3,4}$ School of Mathematical Sciences, Sichuan Normal University, Chengdu 610066, P.R. China

   \vskip 2mm

  $^1$heyao@sicnu.edu.cn,  $^2$huangchushu@stu.sicnu.edu.cn,  $^3$lihaixia@stu.sicnu.edu.cn, $^4$zhangxi@stu.sicnu.edu.cn
\end{center}
\vskip 6mm   {\noindent \bf Abstract.} In this article, we first investigate the partitions whose parts are congruent to $a$ or $b$ modulo $k$ with the aid of separable integer partition classes with modulus $k$ introduced by Andrews. Then, we introduce the $(k,r)$-overpartitions in which only parts equivalent to $r$ modulo $k$ may be overlined and we will show that the number of $(k,k)$-overpartitions of $n$ equals the number of partitions of $n$ such that the $k$-th occurrence of a part may be overlined. Finally, we extend separable integer partition classes with modulus $k$ to overpartitions and then give the generating function for $(k,r)$-modulo overpartitions, which are the $(k,r)$-overpartitions satisfying certain congruence conditions.

\section{Introduction}

A partition $\pi$ of a positive integer $n$ is a finite non-increasing sequence of positive integers $\pi=(\pi_1,\pi_2,\ldots,\pi_m)$ such that $\pi_1+\pi_2+\cdots+\pi_m=n$. The $\pi_i$ are called the parts of $\pi$. Let $\ell(\pi)$ be the number of parts of $\pi$. The weight of $\pi$ is the sum of parts, denoted $|\pi|$.

Throughout this article, we let $k$ be a positive integer. We assume that $a$ and $b$ are integers such that $1\leq a<b\leq k$. We use $\ell_{a}(\pi)$ and $\ell_{b}(\pi)$ to denote the number of parts equivalent to $a$ and $b$ modulo $k$ in a partition $\pi$ respectively.

We define $(a,b,k)$-partition to be the partition with parts equivalent to $a$ or $b$ modulo $k$. Then, $(1,2,3)$-partitions are $3$-regular partitions, in which none of the parts is a multiple of $3$, and $(1,4,5)$-partitions and $(2,3,5)$-partitions are the partitions with congruence conditions in the first and the second Rogers-Ramanujan identities \cite{Ramanujan-1914,Rogers-1894} respectively.

Let $\mathcal{P}_{a,b,k}$ be the set of $(a,b,k)$-partitions. The generating function for the partitions in $\mathcal{P}_{a,b,k}$ is
\[\sum_{\pi\in\mathcal{P}_{a,b,k}}\mu^{\ell_{a}(\pi)}\nu^{\ell_{b}(\pi)}q^{|\pi|}=\frac{1}{(\mu q^a;q^k)_\infty(\nu q^b;q^k)_\infty}.\]
Here and in the sequel, we assume that $|q|<1$ and employ the standard notation \cite{Andrews-1976}:
\[(a;q)_\infty=\prod_{i=0}^{\infty}(1-aq^i) \quad\text{and}\quad (a;q)_n=\frac{(a;q)_\infty}{(aq^n;q)_\infty}.\]

One of the objectives of this article is to give a new generating function for the partitions in $\mathcal{P}_{a,b,k}$ given below with the aid of separable integer partition classes with modulus $k$ introduced by Andrews \cite{Andrews-2022}.
\begin{equation}\label{a-b-gen}
\begin{split}
&\sum_{\pi\in\mathcal{P}_{a,b,k}}\mu^{\ell_{a}(\pi)}\nu^{\ell_{b}(\pi)}q^{|\pi|}\\
&=\sum_{m,h,i\geq 0}\frac{\mu^{m-h-i}\nu^{h+i}q^{ma+kh^2+(b-a)(h+i)}}{(q^k;q^k)_m}{{h+i}\brack{h}}_k{{m-h-i}\brack{h}}_k,
\end{split}
\end{equation}
where ${A\brack B}_k$ is the $q$-binomial coefficient, or Gaussian polynomial for non-negative integers $A$ and $B$ defined as follows:
\[{A\brack B}_k=\left\{\begin{array}{ll}\frac{(q^k;q^k)_A}{(q^k;q^k)_B(q^k;q^k)_{A-B}},&\text{if }A\geq B\geq 0,\\
0,&\text{otherwise.}
\end{array}
\right.\]

In this article, we also investigate the overpartitions with certain congruence conditions.
An overpartition, introduced by Corteel and Lovejoy \cite{Corteel-Lovejoy-2004},  is a partition such that the first occurrence of a part can be overlined. We use $\ell_o(\pi)$ to denote the number of overlined parts in an overpartition $\pi$. In the remaining of this article, we assume that $r$ is an integer such that $k\geq r\geq 1$. An overpartition is called a $(k,r)$-overpartition if only parts equivalent to $r$ modulo $k$ may be overlined. Then, $(1,1)$-overpartitions are overpartitions. For example, there are fifteen $(3,3)$-overpartitions of $6$.
\[(6),(5,1),(4,2),(4,1,1),(3,3),(3,2,1),(3,1,1,1),\]
\[(2,2,2),(2,2,1,1),(2,1,1,1,1),(1,1,1,1,1,1),\]
\[(\bar{6}),(\bar{3},3),(\bar{3},2,1),(\bar{3},1,1,1).\]
Let $\mathcal{O}_{k,r}$ be the set of all $(k,r)$-overpartitions. The generating function for the overpartitions in $\mathcal{O}_{k,r}$ is
\begin{equation}\label{gen-o-kr-eqn}
\sum_{\pi\in\mathcal{O}_{k,r}}x^{\ell_o(\pi)}z^{\ell(\pi)}q^{|\pi|}=\frac{(-xzq^r;q^k)_\infty}{(zq;q)_\infty}.
\end{equation}

A $k$-partition is a partition such that the $k$-th occurrence of a part can be overlined. Then, $1$-partitions are overpartitions. For example, there are fifteen $3$-partitions of $6$.
\[(6),(5,1),(4,2),(4,1,1),(3,3),(3,2,1),(3,1,1,1),\]
\[(2,2,2),(2,2,1,1),(2,1,1,1,1),(1,1,1,1,1,1),\]
\[(3,1,1,\bar{1}),(2,2,\bar{2}),(2,1,1,\bar{1},1),(1,1,\bar{1},1,1,1).\]

Let ${O}_{k,k}(\ell_o,m,n)$ (resp. $P_{k}(\ell_o,m,n)$) denote the number of  $(k,k)$-overpartitions (resp. $k$-partitions) of $n$ with exactly $\ell_o$ overlined parts and $m$ parts. We will show the following theorem.
\begin{theorem}\label{o-p-relation}
For $\ell_o,m,n\geq 0,$ we have
\begin{equation}\label{eqn-o-p}
{O}_{k,k}(\ell_o,m,n)=P_{k}(\ell_o,m+(k-1)\ell_o,n).
\end{equation}
\end{theorem}

For example, it can be checked that
\[{O}_{3,3}(1,m,6)=P_{k}(1,m+2,6)=1\text{ for }1\leq m\leq 4.\]

Then, we focus on a subset of $\mathcal{O}_{k,r}$. Let $\pi=(\pi_1,\pi_2,\ldots,\pi_m)$ be an overpartition in $\mathcal{O}_{k,r}$. For $1\leq i\leq m$, we define
\[\varphi_{k,r}(\pi_i)=s,\text{ where }-k+r+1\leq s\leq r \text{ and }\pi_i\equiv s\pmod{k}.\]
Now, we give the definition of $(k,r)$-modulo overpartitions.
\begin{definition}\label{defi_kr_modulo}
A $(k,r)$-modulo overpartition $\pi$ is a $(k,r)$-overpartition of the form $\pi=(\pi_1,\pi_2,\ldots,\pi_m)$ satisfying the following conditions:
\begin{itemize}
\item[\rm(1)] $\pi_m\equiv 1,2,\ldots,r\pmod{k}$;

\item[\rm(2)] for $1\leq i<m$, if $\varphi_{k,r}(\pi_i)<\varphi_{k,r}(\pi_{i+1})$, then $\pi_{i+1}$ is overlined.
\end{itemize}
\end{definition}

For example, there are eleven $(3,1)$-modulo overpartitions of $6$.
\[(5,\bar{1}),(4,1,1),(\bar{4},1,1),(4,\bar{1},1),(\bar{4},\bar{1},1),(3,2,\bar{1}),(3,\bar{1},1,1),\]
\[(2,2,\bar{1},1),(2,\bar{1},1,1,1),(1,1,1,1,1,1),(\bar{1},1,1,1,1,1).\]

Let $\mathcal{M}_{k,r}$ be the set of all $(k,r)$-modulo overpartitions. We will give the following generating function for the overpartitions in $\mathcal{M}_{k,r}$ with the aid of separable overpartition classes with modulus $k$, which is an extension of separable integer partition classes with modulus $k$ introduced by Andrews \cite{Andrews-2022}.
\begin{equation}\label{gen-m-kr-eqn}
\sum_{\pi\in\mathcal{M}_{k,r}}x^{\ell_o(\pi)}z^{\ell(\pi)}q^{|\pi|}=\sum_{n,j\geq0}\frac{x^jz^{n+j}}{(q^k;q^k)_{n+j}}q^{n+k{j\choose 2}+rj}{{n+kj+r-1}\brack{kj+r-1}}_1.
\end{equation}

This article is organized as follows. In Section 2, we recall the definition of separable integer partition classes with modulus $k$ introduced by Andrews \cite{Andrews-2022} and give a proof of \eqref{a-b-gen}.
Section 3 is devoted to presenting two proofs of Theorem \ref{o-p-relation}, an analytic proof and a combinatorial proof. Finally, we introduce separable overpartition classes with modulus $k$ and give a proof of \eqref{gen-m-kr-eqn} in Section 4.

\section{$(a,b,k)$-partitions}

In this section, we first recall the definition of separable integer partition classes with modulus $k$ introduced by Andrews \cite{Andrews-2022}. Then, we show that $\mathcal{P}_{a,b,k}$ is a separable integer partition class with modulus $k$ and give a equivalent statement of \eqref{a-b-gen}, which is given in \eqref{a-b-g-gen}. Finally, we give a proof of \eqref{a-b-g-gen}.

Andrews \cite{Andrews-2022} first introduced separable integer partition classes with modulus $k$ and analyzed some well-known theorems, such as the first G\"ollnitz-Gordon identity, Schur's partition theorem, partitions  with $n$ copies of $n$, and so on. Based on separable integer partition classes with modulus $k$, Passary \cite[Section 3]{Passary-2019} studied two cases of partitions with parts separated by parity, little G\"ollnitz identities and the second G\"ollnitz-Gordon identity, and Chen, He, Tang and Wei  \cite[Section 3]{Chen-He-Tang-Wei-2024} investigated the remaining six cases of partitions with parts separated by parity.

 \begin{definition}
A separable integer partition class $\mathcal{P}$ with modulus $k$ is a subset of all the partitions satisfying the following{\rm:}

There is a subset $\mathcal{B}$ of $\mathcal{P}$ {\rm(}$\mathcal{B}$ is called the basis of $\mathcal{P}${\rm)} such that for each integer $m\geq 1$, the number of partitions in $\mathcal{B}$ with $m$ parts is finite and every partition in $\mathcal{P}$ with $m$ parts is uniquely of the form
\begin{equation}\label{ordinary-form-1}
(\lambda_1+\mu_1)+(\lambda_2+\mu_2)+\cdots+(\lambda_m+\mu_m),
\end{equation}
where $(\lambda_1,\lambda_2,\ldots,\lambda_m)$ is a partition in $\mathcal{B}$ and $(\mu_1,\mu_2,\ldots,\mu_m)$ is a  non-increasing sequence of nonnegative integers, whose only restriction is that each part is divisible by $k$. Furthermore, all partitions of the form \eqref{ordinary-form-1} are in $\mathcal{P}$.
\end{definition}

For $m\geq 1$, let $\mathcal{B}_{a,b,k}(m)$ be the set of partitions in $\mathcal{P}_{a,b,k}$ with $m$ parts of the form $\lambda=(\lambda_{1},\lambda_{2},\ldots \lambda_{m})$ satisfying the following conditions:
\begin{itemize}
\item[(1)] $\lambda_{m}=a$ or $b$;
\item[(2)]  for $1\leq i<m$, $\lambda_{i}<\lambda_{i+1}+k$.
\end{itemize}

 For $m\geq 1$, assume that $\lambda=(\lambda_{1},\lambda_{2},\ldots \lambda_{m})$ is a partition in $\mathcal{B}_{a,b,k}(m)$. For $1\leq i<m$, if $\lambda_{i+1}=kh+a$, then we have $kh+a\leq \lambda_i<k(h+1)+a$, and so $\lambda_i=kh+a$ or $kh+b$; if $\lambda_{i+1}=kh+b$, then we have $kh+b\leq \lambda_i<k(h+1)+b$, and so $\lambda_i=kh+b$ or $k(h+1)+a$. Therefore, we see that there are $2^m$ partitions in $\mathcal{B}_{a,b,k}(m)$.

For example, the number of partitions in $\mathcal{B}_{a,b,k}(3)$ is $2^3=8$.
\[(a,a,a),(b,a,a),(b,b,a),(k+a,b,a),\]
\[(b,b,b),(k+a,b,b),(k+a,k+a,b),(k+b,k+a,b).\]

Set
\[\mathcal{B}_{a,b,k}=\bigcup_{m\geq1}\mathcal{B}_{a,b,k}(m).\]
 Obviously, $\mathcal{B}_{a,b,k}$ is the basis of $\mathcal{P}_{a,b,k}$. So, we have
 \begin{theorem}
$\mathcal{P}_{a,b,k}$ is a separable integer partition class with modulus $k$.
\end{theorem}
For $m\geq 1$, we define \[g_{a,b,k}(m)=\sum_{\lambda\in\mathcal{B}_{a,b,k}(m)}\mu^{\ell_{a}(\lambda)}\nu^{\ell_{b}(\lambda)}q^{|\lambda|}.\]
Then, we have
\[\sum_{\pi\in\mathcal{P}_{a,b,k}}\mu^{\ell_{a}(\pi)}\nu^{\ell_{b}(\pi)}q^{|\pi|}=1+\sum_{m\geq 1}\frac{g_{a,b,k}(m)}{(q^k;q^k)_m}.\]
We find that \eqref{a-b-gen} is equivalent to the following identity.
\begin{equation}\label{a-b-g-gen}
g_{a,b,k}(m)=\sum_{h,i\geq 0}\mu^{m-h-i}\nu^{h+i}q^{ma+kh^2+(b-a)(h+i)}{{h+i}\brack{h}}_k{{m-h-i}\brack{h}}_k.
\end{equation}
In order to show \eqref{a-b-g-gen}, we need to give the generating functions for the partitions in $\mathcal{B}_{a,b,k}(m)$ with the largest part $l$, denoted ${g}_{a,b,k}(m,l)$.

\begin{lemma}\label{a-b-m-h-lem}
For $m\geq 1$, we have
\begin{equation}\label{eqn-h-1}
g_{a,b,k}(m,a)=\mu^{m}q^{ma}.
\end{equation}
For $m\geq 1$ and $h\geq 1$, we have
\begin{equation}\label{eqn-h-2}
\begin{split}
&g_{a,b,k}(m,kh+a)\\
&=\sum_{i\geq 0}\mu^{m-h-i}\nu^{h+i}q^{ma+kh^2+(b-a)(h+i)}{{h+i-1}\brack{h-1}}_k{{m-h-i}\brack{h}}_k.
\end{split}
\end{equation}
For $m\geq 1$ and $h\geq 0$, we have
\begin{equation}\label{eqn-h-3}
\begin{split}
&g_{a,b,k}(m,kh+b)\\
&=\sum_{i\geq 0}\mu^{m-h-i-1}\nu^{h+i+1}q^{ma+kh^2+kh+(b-a)(h+i+1)}{{h+i}\brack{h}}_k{{m-h-i-1}\brack{h}}_k.
\end{split}
\end{equation}
\end{lemma}

\begin{proof}
It is clear that for $m\geq 1$, there is only one partition
\[(\underbrace{a,a,\ldots,a}_{m's})\]
in $\mathcal{B}_{a,b,k}(m)$ with the largest part $a$, which leads to \eqref{eqn-h-1}.

Then, we proceed to show \eqref{eqn-h-2}. For $m\geq 1$ and $h\geq 1$, let $\lambda$ be a partition in  $\mathcal{B}_{a,b,k}(m)$ with the largest part ${kh+a}$. There exist parts
\[kh+a,k(h-1)+b,k(h-1)+a,k(h-2)+b,\ldots,k+a,b\]
in $\lambda$. We remove one $kh+a$, one $k(h-1)+b$, \ldots, one $k+a$ and one $b$ from $\lambda$ and denote the resulting partition by $\alpha$. Clearly, there are $m-2h$ parts in $\alpha$ and the parts of $\alpha$ do not exceed ${kh+a}$. Let $\alpha^{(a)}$ (resp. $\alpha^{(b)}$) be the partition consisting of the parts equivalent to $a$ (resp. $b$) modulo $k$ in  $\alpha$. Then, the parts in $\alpha^{(a)}$ (resp. $\alpha^{(b)}$) do not exceed ${kh+a}$ (resp. ${k(h-1)+b}$). Assume that $\ell(\alpha^{(b)})=i$, then we have $\ell(\alpha^{(a)})=m-2h-i$. The process above to get $\alpha^{(a)}$ and $\alpha^{(b)}$ could be run in reverse.

The generating function for the partitions with $i$ parts not exceeding ${k(h-1)+b}$ and equivalent to $b$ modulo $k$ is
\[q^{bi}{{h+i-1}\brack{h-1}}_k.\]
   The generating function for the partitions with $m-2h-i$ parts not exceeding ${kh+a}$ and equivalent to $a$ modulo $k$ is
   \[q^{a(m-2h-i)}{{m-h-i}\brack{h}}_k.\]
So, we get
\begin{align*}
g_{a,b,k}(m,kh+a)&=\mu^{h}q^{k{{h+1}\choose 2}+ah}\nu^{h}q^{k{{h}\choose 2}+bh}\\
&\quad\times\sum_{i\geq 0}\nu^{i}q^{bi}{{h+i-1}\brack{h-1}}_k\mu^{m-2h-i}q^{a(m-2h-i)}{{m-h-i}\brack{h}}_k.
\end{align*}
Hence,  \eqref{eqn-h-2} is verified.

 With a similar argument above, we get \eqref{eqn-h-3}.  The proof is complete.
\end{proof}

We also need the following recurrence.
\begin{lemma}\label{a-b-m-h-lem-0}
For $m\geq 1$ and $h\geq 0$, we have
\begin{equation}\label{a-b-g-rec}
g_{a,b,k}(m+1,kh+b)=\nu q^{kh+b}(g_{a,b,k}(m,kh+a)+g_{a,b,k}(m,kh+b)).
\end{equation}
\end{lemma}
 We will present an analytic proof and a combinatorial proof of Lemma \ref{a-b-m-h-lem-0}. We first give an analytic proof of Lemma \ref{a-b-m-h-lem-0}.
 \begin{proof}[Analytic proof of Lemma \ref{a-b-m-h-lem-0}] For $h=0$, by Lemma \ref{a-b-m-h-lem}, we have
 \begin{align*}
 g_{a,b,k}(m,a)+g_{a,b,k}(m,b)&=\mu^{m}q^{ma}+\sum_{i=0}^{m-1}\mu^{m-i-1}\nu^{i+1}q^{ma+(b-a)(i+1)}\\
 &=\sum_{i=-1}^{m-1}\mu^{m-i-1}\nu^{i+1}q^{ma+(b-a)(i+1)}\\
 &=\sum_{i=0}^{m}\mu^{m-i}\nu^{i}q^{ma+(b-a)i}\\
 &=(\nu q^b)^{-1}\sum_{i=0}^{m}\mu^{(m+1)-i-1}\nu^{i+1}q^{(m+1)a+(b-a)(i+1)}\\
 &=(\nu q^b)^{-1}g_{a,b,k}(m+1,b).
 \end{align*}
 So, \eqref{a-b-g-rec} holds for $h=0$.

  For $h\geq 1$, again by Lemma \ref{a-b-m-h-lem}, we have
 \begin{align*}
 &\quad g_{a,b,k}(m,kh+a)+g_{a,b,k}(m,kh+b)\\
 &=\mu^{m-h}\nu^{h}q^{ma+kh^2+(b-a)h}{{m-h}\brack{h}}_k\\
 &\qquad+\sum_{i\geq 1}\mu^{m-h-i}\nu^{h+i}q^{ma+kh^2+(b-a)(h+i)}{{h+i-1}\brack{h-1}}_k{{m-h-i}\brack{h}}_k\\
 &\quad\qquad+\sum_{i\geq 1}\mu^{m-h-i}\nu^{h+i}q^{ma+kh^2+kh+(b-a)(h+i)}{{h+i-1}\brack{h}}_k{{m-h-i}\brack{h}}_k\\
 &=\mu^{m-h}\nu^{h}q^{ma+kh^2+(b-a)h}{{m-h}\brack{h}}_k\\
 &\quad+\sum_{i\geq 1}\mu^{m-h-i}\nu^{h+i}q^{ma+kh^2+(b-a)(h+i)}{{m-h-i}\brack{h}}_k\\
 &\qquad\qquad\times\left\{{{h+i-1}\brack{h-1}}_k+q^{kh}{{h+i-1}\brack{h}}_k\right\}.\\
 \end{align*}
 Combining with the standard recurrence for the $q$-binomial coefficients \cite[(3.3.4)]{Andrews-1976}:
\begin{equation}\label{bin-new-r-1}
{A\brack B}_k={{A-1}\brack{B-1}}_k+q^{kB}{{A-1}\brack{B}}_k,
\end{equation}
 we get
 \begin{align*}
 &\quad g_{a,b,k}(m,kh+a)+g_{a,b,k}(m,kh+b)\\
 &=\mu^{m-h}\nu^{h}q^{ma+kh^2+(b-a)h}{{m-h}\brack{h}}_k\\
 &\quad+\sum_{i\geq 1}\mu^{m-h-i}\nu^{h+i}q^{ma+kh^2+(b-a)(h+i)}{{m-h-i}\brack{h}}_k{{h+i}\brack{h}}_k\\
 &=\sum_{i\geq 0}\mu^{m-h-i}\nu^{h+i}q^{ma+kh^2+(b-a)(h+i)}{{m-h-i}\brack{h}}_k{{h+i}\brack{h}}_k\\
 &=(\nu q^{kh+b})^{-1}\sum_{i\geq 0}\mu^{m-h-i}\nu^{h+i+1}q^{(m+1)a+kh^2+kh+(b-a)(h+i+1)}{{m-h-i}\brack{h}}_k{{h+i}\brack{h}}_k\\
 &=(\nu q^{kh+b})^{-1}g_{a,b,k}(m+1,kh+b).
 \end{align*}
We arrive at \eqref{a-b-g-rec} for $h\geq 1$. This completes the proof.
 \end{proof}

Then, we give a combinatorial proof of Lemma \ref{a-b-m-h-lem-0}.
 \begin{proof}[Combinatorial proof of Lemma \ref{a-b-m-h-lem-0}]
Let $\lambda=(\lambda_1,\lambda_2,\ldots,\lambda_{m+1})$ be a partition in $\mathcal{B}_{a,b,k}(m+1)$ with the largest part ${kh+b}$. Then, we have $\lambda_1=kh+b$. By the definition of $\mathcal{B}_{a,b,k}(m+1)$, we get
 $\lambda_2>\lambda_1-k=k(h-1)+b$, and so $\lambda_2=kh+b$ or $kh+a$. If we remove the largest part $kh+b$ from $\lambda$, then we can get a partition in $\mathcal{B}_{a,b,k}(m)$ with the largest part $kh+b$ or $kh+a$, and vice versa. This implies that \eqref{a-b-g-rec} holds. The proof is complete.
 \end{proof}

 We are now in a position to give a proof of \eqref{a-b-g-gen}.
 \begin{proof}[Proof of \eqref{a-b-g-gen}]
  Combining Lemmas \ref{a-b-m-h-lem} and \ref{a-b-m-h-lem-0}, we obtain that for $m\geq 1$,
 \begin{align*}
 g_{a,b,k}(m)&=\sum_{h\geq 0}(g_{a,b,k}(m,kh+a)+g_{a,b,k}(m,kh+b))\\
 &=\sum_{h\geq 0}(\nu q^{kh+b})^{-1}g_{a,b,k}(m+1,kh+b)\\
 &=\sum_{h\geq 0}\sum_{i\geq 0}\mu^{m-h-i}\nu^{h+i}q^{ma+kh^2+(b-a)(h+i)}{{h+i}\brack{h}}_k{{m-h-i}\brack{h}}_k.
 \end{align*}
 We arrive at \eqref{a-b-g-gen}. The proof is complete.
 \end{proof}

\section{Proofs of Theorem \ref{o-p-relation}}

The main objective of this section is to prove Theorem \ref{o-p-relation}. We will present two proofs, an analytic proof and a combinatorial proof. We first give an analytic proof of Theorem \ref{o-p-relation}.

\begin{proof}[Analytic proof of Theorem \ref{o-p-relation}] It is clear from \eqref{gen-o-kr-eqn} that

\begin{equation}\label{gen-o-kk-eqn}
\sum_{\ell_o,m,n\geq 0}{O}_{k,k}(\ell_o,m,n)x^{\ell_o}z^mq^n=\frac{(-xzq^k;q^k)_\infty}{(zq;q)_\infty}.
\end{equation}

The generating function for $P_{k}(\ell_o,m,n)$ is
\begin{equation}\label{gen-p-kk-eqn}
\begin{split}
&\quad\sum_{\ell_o,m,n\geq 0}{P}_{k}(\ell_o,m,n)x^{\ell_o}z^mq^n\\
&=
\prod_{j=1}^{\infty}\left(1+zq^j+\cdots+z^{k-1}q^{(k-1)j}+(1+x)z^{k}q^{kj}+(1+x)z^{k+1}q^{(k+1)j}+\cdots\right)\\
&=\prod_{j=1}^{\infty}\left(\frac{1-z^kq^{kj}}{1-zq^j}+\frac{(1+x)z^{k}q^{kj}}{1-zq^j}\right)\\
&=\prod_{j=1}^{\infty}\frac{1+xz^{k}q^{kj}}{1-zq^j}\\
&=\frac{(-xz^{k}q^k;q^k)_\infty}{(zq;q)_\infty}.
\end{split}
\end{equation}

Letting $x\rightarrow xz^{-(k-1)}$ in \eqref{gen-p-kk-eqn}, we get
\begin{equation*}\label{gen-p-kk-1-eqn}
\quad\sum_{\ell_o,m,n\geq 0}P_{k}(\ell_o,m+(k-1)\ell_o,n)x^{\ell_o}z^mq^n=\frac{(-xzq^k;q^k)_\infty}{(zq;q)_\infty}.
\end{equation*}
Combining with \eqref{gen-o-kk-eqn}, we arrive at \eqref{eqn-o-p}. This completes the proof.
\end{proof}

Then, we give a combinatorial proof of Theorem \ref{o-p-relation}.

{\noindent \it Combinatorial proof of Theorem \ref{o-p-relation}:} Let $\pi=(\pi_1,\pi_2,\ldots,\pi_m)$ be an overpartition counted by ${O}_{k,k}(\ell_o,m,n)$. Assume that $\overline{k\zeta_1}$, $\overline{k\zeta_2}$,\ldots, $\overline{k\zeta_{\ell_o}}$ are the $\ell_o$ overlined parts in $\pi$, where $\zeta_1>\zeta_2>\cdots>\zeta_{\ell_o}\geq 1$. We first remove $\overline{k\zeta_1}$, $\overline{k\zeta_2}$,\ldots, $\overline{k\zeta_{\ell_o}}$ from $\pi$ and then add
\[\underbrace{\zeta_1,\ldots,\zeta_1,}_{(k-1)'s}\overline{\zeta_1},\underbrace{\zeta_2,\ldots,\zeta_2,}_{(k-1)'s}\overline{\zeta_2},\ldots,\underbrace{\zeta_{\ell_o},\ldots,\zeta_{\ell_o},}_{(k-1)'s}\overline{\zeta_{\ell_o}}\]
as new parts into the resulting partition to get a partition $\lambda$.
Clearly, $\lambda$ is a partition enumerated by $P_{k}(\ell_o,m+(k-1)\ell_o,n)$. Obviously, the process above is reversible. The proof is complete.  \qed

For example, let $\pi=(\bar{9},7,6,6,5,\bar{3},3,1,1)$ be an overpartition enumerated by ${O}_{3,3}(2,9,41)$. We
first remove $\bar{9}$ and $\bar{3}$ from $\pi$ and then add $3,3,\bar{3},1,1,\bar{1}$ as new parts into the resulting partition to get $\lambda=(7,6,6,5,3,3,\bar{3},3,1,1,\bar{1},1,1)$. Clearly, $\lambda$ is a partition enumerated by $P_{3}(2,13,41)$. Moreover, the process to get $\lambda$ could be run in reverse.

\section{$(k,r)$-modulo overpartitions}

In this section, we first introduce separable overpartition classes with modulus $k$. Then, we show that $\mathcal{M}_{k,r}$ is a separable overpartition class with modulus $k$. Finally, we give a proof of \eqref{gen-m-kr-eqn}.

\subsection{Separable overpartition classes with modulus $k$}

 In \cite{Chen-He-Tang-Wei-2024}, Chen, He, Tang and Wei extended separable integer partition classes with modulus $1$ introduced by Andrews \cite{Andrews-2022} to overpartitions, called separable overpartition classes, and then studied overpartitions and the overpartition analogue of Rogers-Ramanujan identities from the view of separable overpartition classes.

 In this article, we extend separable integer partition classes with modulus $k$ to overpartitions, which is stated as follows. Here and in the sequel,
we adopt the following convention: For positive integer $t$ and non-negative integer $b$, we define
\[\overline{t}\pm b=\overline{t\pm b}\text{ and }\overline{t}-\overline{b}=t-b.\]

 \begin{definition}\label{defi-separable}
A subset $\mathcal{P}$ of all the overpartitions is called a separable overpartition class with modulus $k$ if
$\mathcal{P}$   satisfies the following{\rm:}

There is a subset $\mathcal{B}$ of $\mathcal{P}$ {\rm(}$\mathcal{B}$ is called the basis of $\mathcal{P}${\rm)} such that for each integer $m\geq 1$,
 \begin{itemize}
\item[\rm{(1)}]  the number of overpartitions in $\mathcal{B}$ with $m$ parts is finite{\rm;}
\item[\rm{(2)}] every overpartition in $\mathcal{P}$ with $m$ parts is uniquely of the form
\begin{equation}\label{over-form-1}
(\lambda_1+\mu_1)+(\lambda_2+\mu_2)+\cdots+(\lambda_m+\mu_m),
\end{equation}
where $(\lambda_1,\lambda_2,\ldots,\lambda_m)$ is an overpartition in $\mathcal{B}$ and $(\mu_1,\mu_2,\ldots,\mu_m)$ is a  non-increasing sequence of nonnegative integers, whose only restriction is that each part is divisible by $k${\rm;}
\item[\rm{(3)}] all overpartitions of the form \eqref{over-form-1} are in $\mathcal{P}$.

\end{itemize}
\end{definition}

Assume that $\mathcal{P}$ is a separable overpartition class with modulus $k$ and  $\mathcal{B}$ is the basis of $\mathcal{P}$. Let $\pi=(\pi_1,\pi_2,\ldots,\pi_m)$ be an overpartition in $\mathcal{P}$.
Then, there exist unique overpartition $\lambda=(\lambda_1,\lambda_2,\ldots,\lambda_m)$ in $\mathcal{B}$ and non-increasing sequence  $(\mu_1,\mu_2,\ldots,\mu_m)$ of nonnegative integers divisible by $k$ such that $\pi_i=\lambda_i+\mu_i$ for $1\leq i\leq m$.
Clearly, the number of overlined parts in $\pi$ equals the number of overlined parts in $\lambda$.

For $m\geq 1$, let $\mathcal{B}(m)$ be the set of overpartitions in $\mathcal{B}$ with  $m$ parts. Then, the generating function for  the overpartitions in $\mathcal{P}$ is
 \begin{equation}\label{useful}
  \sum_{\pi\in\mathcal{P}}x^{\ell_o(\pi)}z^{\ell(\pi)}q^{|\pi|}=1+\sum_{m\geq 1}\frac{z^m}{(q^k;q^k)_m}\sum_{\lambda\in\mathcal{B}(m)}x^{\ell_o(\lambda)}q^{|\lambda|}.
  \end{equation}

 \subsection{The basis of $\mathcal{M}_{k,r}$}

 The objective of this subsection is to show that $\mathcal{M}_{k,r}$ is a separable overpartition class with modulus $k$. To do this, we are required to find the basis of $\mathcal{M}_{k,r}$, which involves the following set. We impose the following order on the parts of an overpartition:
\[1<\bar{1}<2<\bar{2}<\cdots.\]

 \begin{definition}\label{base-M_kr}
For $m\geq 1$, define $\mathcal{B}_{k,r}(m)$ to be the set of overpartitions of the form $\lambda=(\lambda_{1},\lambda_{2},\ldots,\lambda_{m})$ satisfying the following conditions:
\begin{itemize}
\item[{\rm(1)}] Only parts equivalent to $r$ modulo $k$ may be overlined{\rm;}
\item[{\rm(2)}] $\lambda_{m}\leq\bar{r}${\rm;}
\item[{\rm(3)}] for $1\leq i<m,$ if $\overline{k(j-1)+r}\leq\lambda_{i+1}\leq {kj+r}$ then we have $\lambda_{i}\leq \overline{kj+r}$.
\end{itemize}
\end{definition}

For example, there are nineteen overpartitions in $\mathcal{B}_{3,1}(3)$.
\[(1,1,1),(\bar{1},1,1),(2,\bar{1},1),(3,\bar{1},1),(4,\bar{1},1),(\bar{4},\bar{1},1),\]
\[(2,2,\bar{1}),(3,2,\bar{1}),(4,2,\bar{1}),(\bar{4},2,\bar{1}),(3,3,\bar{1}),(4,3,\bar{1}),(\bar{4},3,\bar{1}),\]
\[(4,4,\bar{1}),(\bar{4},4,\bar{1}),(5,\bar{4},\bar{1}),(6,\bar{4},\bar{1}),(7,\bar{4},\bar{1}),(\bar{7},\bar{4},\bar{1}).\]

Set
\[\mathcal{B}_{k,r}=\bigcup_{m\geq 1}\mathcal{B}_{k,r}(m).\]
We proceed to show that $\mathcal{M}_{k,r}$ is a separable overpartition class with modulus $k$.

\begin{theorem}\label{m-b-class}
$\mathcal{M}_{k,r}$ is a separable overpartition class with modulus $k$ and  $\mathcal{B}_{k,r}$ is the basis of $\mathcal{M}_{k,r}$.
\end{theorem}

\begin{proof} We just need to show that $\mathcal{B}_{k,r}$ is the basis of $\mathcal{M}_{k,r}$. Clearly, $\mathcal{B}_{k,r}$ is a subset of $\mathcal{M}_{k,r}$. It remains to prove that $\mathcal{B}=\mathcal{B}_{k,r}$ and $\mathcal{M}=\mathcal{M}_{k,r}$ satisfies the conditions (1)-(3) in Definition \ref{defi-separable}.

 By the conditions (2) and (3) in Definition \ref{base-M_kr}, we find that for $m\geq 1$, the number of overpartitions in $\mathcal{B}_{k,r}(m)$ does not exceed $(r+1)\times(k+1)^{m-1}$, which yields that the number of overpartitions in $\mathcal{B}_{k,r}$ with $m$ parts is finite. So, $\mathcal{B}_{k,r}$ satisfies the condition (1) in Definition \ref{defi-separable}.

 We proceed to show that $\mathcal{B}_{k,r}$ and $\mathcal{M}_{k,r}$ satisfies the conditions (2) in Definition \ref{defi-separable}. For $m\geq 1$, let $\pi=(\pi_1,\pi_2,\ldots,\pi_m)$ be an overpartition in $\mathcal{M}_{k,r}$ with $m$ parts. There exists an unique overpartition $(\lambda_1,\lambda_2,\ldots,\lambda_m)$ in $\mathcal{B}_{k,r}$ such that $\pi_i\equiv \lambda_i\pmod{k}$ for $1\leq i\leq m$. Moreover precisely, assume that there are $j$ overlined parts in $\pi$, we consider the following two cases.

 Case 1: $j=0$. In this case, we have $\lambda_1\leq r$.

 Case 2: $j\geq 1$. Assume that $\pi_{i_1}<\pi_{i_2}<\cdots<\pi_{i_j}$ are the $j$ overlined parts in $\pi$. Then, we have $\lambda_{i_1}=\overline{r}$, $\lambda_{i_2}=\overline{k+r}$,\ldots, $\lambda_{i_j}=\overline{k(j-1)+r}$.

 In either case, we see that for $1\leq i\leq m$, $\pi_i$ and $\lambda_i$ are both overlined or both non-overlined. For $1\leq i\leq m$, we set $\mu_i=\pi_i-\lambda_i$.
Clearly, $(\mu_1,\mu_2,\ldots,\mu_m)$ is a  non-increasing sequence of nonnegative integers  divisible by $k$.
This implies that $\mathcal{B}_{k,r}$ and $\mathcal{M}_{k,r}$ satisfies the conditions (2) in Definition \ref{defi-separable}.

It is clear that $\mathcal{B}_{k,r}$ and $\mathcal{M}_{k,r}$ satisfies the conditions (3) in Definition \ref{defi-separable}. Now, we can conclude that $\mathcal{B}_{k,r}$ is the basis of $\mathcal{M}_{k,r}$. This completes the proof.
\end{proof}

\subsection{Proof of \eqref{gen-m-kr-eqn}}

In this subsection, we aim to give a proof of \eqref{gen-m-kr-eqn}. To do this, we first give the generating functions for the overpartitions in $\mathcal{B}_{k,r}(m)$ with the largest part $l$ (resp. $\bar{l}$), denoted ${g}_{k,r}(m,l)$ (resp. ${g}_{k,r}(m,\bar{l})$).

\begin{lemma}\label{bkrj-lem}
Assume that $m\geq 1$. For $j\geq 1$, we have
\[
g_{k,r}(m,\overline{k(j-1)+r})=q^{m-j+k{j\choose 2}+rj}{{m-j+k(j-1)+r-1}\brack{k(j-1)+r-1}}_1.
\]
For $j=1$ and $1\leq s\leq r$, or $j\geq 2$ and $-k+r+1\leq s\leq r$, we have
\begin{equation}\label{bkrjs-eqn}
g_{k,r}(m,{k(j-1)+s})=q^{m-j+k{j\choose 2}+r(j-1)+s}{{m-j+k(j-1)+s-1}\brack{k(j-1)+s-1}}_1.
\end{equation}
\end{lemma}

\begin{proof} We first show that for $j\geq 1$,
   \begin{equation}\label{overlinerr}
   g_{k,r}(m,\overline{k(j-1)+r})=g_{k,r}(m,{k(j-1)+r}).
   \end{equation}

   For an overpartition $\lambda$ in $\mathcal{B}_{k,r}(m)$ with the largest part $\overline{k(j-1)+r}$, we can get an overpartition in $\mathcal{B}_{k,r}(m)$ with the largest part ${k(j-1)+r}$ by changing the overlined part $\overline{k(j-1)+r}$ in $\lambda$ to a non-overlined part ${k(j-1)+r}$, and vice versa.
   This implies that \eqref{overlinerr} holds.

   It remains to show that \eqref{bkrjs-eqn} is valid. Let $\lambda$ be an overpartition in  $\mathcal{B}_{k,r}(m)$ with the largest part ${k(j-1)+s}$. We remove some parts from $\lambda$. There are two cases.

   Case 1: $j=1$ and $1\leq s\leq r$. In this case, there are no overlined parts in $\lambda$. We remove one $s$ from $\lambda$ to get $\alpha$ and set $\beta=(s)$.

   Case 2: $j\geq 2$ and $-k+r+1\leq s\leq r$. In this case, there are $j-1$ overlined parts in $\lambda$, which are $\overline{k(j-2)+r}$,\ldots, $\overline{k+r}$, $\overline{r}$. We remove the $j-1$ overlined parts and one ${k(j-1)+s}$ from $\lambda$ to get $\alpha$ and set $\beta=({k(j-1)+s},\overline{k(j-2)+r},\ldots,\overline{k+r},\overline{r})$.

   In conclusion, $|\beta|=k{j\choose 2}+r(j-1)+s$, there are $m-j$ parts in $\alpha$, there are no overlined parts in $\alpha$ and the parts of $\alpha$ do not exceed ${k(j-1)+s}$. We delete the parts $1$ in $\alpha$ and subtract one from the remaining parts of $\alpha$ to get $\gamma$. Then, $\gamma$ is a partition with at most $m-j$ parts not exceeding ${k(j-1)+s-1}$. Clearly, the process above to get $\beta$ and $\gamma$ could be run in reverse.

   The generating function for the partitions with at most $m-j$ parts not exceeding ${k(j-1)+s-1}$ is
   \[{{m-j+k(j-1)+s-1}\brack{k(j-1)+s-1}}_1.\]
   So, we get
   \[g_{k,r}(m,{k(j-1)+s})=q^{k{j\choose 2}+r(j-1)+s}\times q^{m-j}{{m-j+k(j-1)+s-1}\brack{k(j-1)+s-1}}_1.\]
   Hence,  \eqref{bkrjs-eqn} is verified. The proof is complete.
   \end{proof}

 We also need the following recurrences.

  \begin{lemma}\label{recurrence-lem}
  For $m\geq 1$, we have
  \begin{equation}\label{recurrence-eqn-1}
  g_{k,r}(m+1,\overline{r})=q^r\sum_{s=1}^{r}g_{k,r}(m,s).
  \end{equation}
  For $m\geq 1$ and $j\geq1$,  we have
   \begin{equation}\label{recurrence-eqn-2}
  g_{k,r}(m+1,\overline{kj+r})=q^{kj+r}\left(g_{k,r}(m,\overline{k(j-1)+r})+\sum_{s=-k+r+1}^{r}g_{k,r}(m,{kj+s})\right).
  \end{equation}
  \end{lemma}

  We will present an analytic proof and a combinatorial proof of Lemma \ref{recurrence-lem}. We first give an analytic proof of Lemma \ref{recurrence-lem}.

  \begin{proof}[Analytic proof of Lemma \ref{recurrence-lem}.] Combining Lemma \ref{bkrj-lem} and the following recurrence for the $q$-binomial coefficients \cite[(3.3.9)]{Andrews-1976}:
\[{{A+B+1}\brack{B+1}}_1=\sum_{s=0}^{A}q^s{{B+s}\brack{B}}_1\text{ for }A,B\geq 0,\]
   we obtain that for $m\geq 1$,
\begin{align*}
\sum_{s=1}^{r}g_{k,r}(m,s)&=\sum_{s=1}^{r}q^{m-1+s}{{m-1+s-1}\brack {s-1}}_1\\
&=q^m\sum_{s=0}^{r-1}q^{s}{{m-1+s}\brack {s}}_1\\
&=q^m\sum_{s=0}^{r-1}q^{s}{{m-1+s}\brack {m-1}}_1\\
&=q^m{{m+r-1}\brack {m}}_1\\
&=q^{-r}q^{(m+1)-1+r}{{(m+1)-1+r-1}\brack {r-1}}_1\\
&=q^{-r}g_{k,r}(m+1,\overline{r}).
\end{align*}
So, \eqref{recurrence-eqn-1} is verified.

Then, we proceed to show \eqref{recurrence-eqn-2}. Using Lemma \ref{bkrj-lem}, we obtain that  for $m\geq 1$ and $j\geq 1$,
\begin{align*}
&\quad g_{k,r}(m,\overline{k(j-1)+r})+\sum_{s=-k+r+1}^{r}g_{k,r}(m,{kj+s})\nonumber\\
&=q^{m-j+k{j\choose 2}+rj}{{m-j+k(j-1)+r-1}\brack{k(j-1)+r-1}}_1+\nonumber\\
&\qquad +\sum_{s=-k+r+1}^{r}q^{m-(j+1)+k{{j+1}\choose 2}+rj+s}{{m-(j+1)+kj+s-1}\brack{kj+s-1}}_1\nonumber\\
&=q^{m-j+k{j\choose 2}+rj}\left\{{{m-j+k(j-1)+r-1}\brack{k(j-1)+r-1}}_1\right.\nonumber\\
&\qquad+q^{k(j-1)+r}{{m-j+k(j-1)+r-1}\brack{k(j-1)+r}}_1+q^{k(j-1)+r+1}{{m-j+k(j-1)+r}\brack{k(j-1)+r+1}}_1\nonumber\\
&\qquad\qquad\left.+\cdots+q^{kj+r-1}{{m-j+kj+r-2}\brack{kj+r-1}}_1\right\}
\end{align*}

By successively applying the standard recurrence for the $q$-binomial coefficients in \eqref{bin-new-r-1}, we get
\begin{align*}
&\quad{{m-j+k(j-1)+r-1}\brack{k(j-1)+r-1}}_1+q^{k(j-1)+r}{{m-j+k(j-1)+r-1}\brack{k(j-1)+r}}_1\nonumber\\
&\qquad+q^{k(j-1)+r+1}{{m-j+k(j-1)+r}\brack{k(j-1)+r+1}}_1+\cdots+q^{kj+r-1}{{m-j+kj+r-2}\brack{kj+r-1}}_1\nonumber\\
&={{m-j+k(j-1)+r}\brack{k(j-1)+r}}_1+q^{k(j-1)+r+1}{{m-j+k(j-1)+r}\brack{k(j-1)+r+1}}_1\nonumber\\
&\qquad+\cdots+q^{kj+r-1}{{m-j+kj+r-2}\brack{kj+r-1}}_1\nonumber\\
&=\cdots\\
&={{m-j+kj+r-1}\brack{kj+r-1}}_1.
\end{align*}

So, we have
\begin{align*}
&\quad g_{k,r}(m,\overline{k(j-1)+r})+\sum_{s=-k+r+1}^{r}g_{k,r}(m,{kj+s})\nonumber\\
&=q^{m-j+k{j\choose 2}+rj}{{m-j+kj+r-1}\brack{kj+r-1}}_1\nonumber\\
&=q^{-kj-r}q^{(m+1)-(j+1)+k{{j+1}\choose 2}+r(j+1)}{{(m+1)-(j+1)+kj+r-1}\brack{kj+r-1}}_1\nonumber\\
&=q^{-kj-r}g_{k,r}(m+1,\overline{kj+r}).\nonumber
\end{align*}

Hence, we arrive at  \eqref{recurrence-eqn-2}. This completes the proof.
\end{proof}

Then, we give a combinatorial proof of Lemma \ref{recurrence-lem}.

 \begin{proof}[Combinatorial proof of Lemma \ref{recurrence-lem}]  Assume that $\lambda=(\lambda_1,\lambda_2,\ldots,\lambda_{m+1})$ is an overpartition in $\mathcal{B}_{k,r}(m+1)$ with the largest part $\overline{kj+r}$, where $j\geq 0$. Then, we have $\lambda_1=\overline{kj+r}$. We consider the following two cases.

 Case 1: $j=0$. In this case, we have $\lambda_1=\overline{r}$. By the condition (3) in Definition \ref{base-M_kr}, we deduce that $1\leq \lambda_2\leq r$.  If we remove the largest part $\overline{r}$ from $\lambda$, then we can get an overpartition in $\mathcal{B}_{k,r}(m)$ with the largest part not exceeding $r$, and vice versa. This implies that \eqref{recurrence-eqn-1} holds.

 Case 2: $j\geq 1$. It is from the condition (3) in Definition \ref{base-M_kr} that
 $\overline{k(j-1)+r}\leq \lambda_2\leq {kj+r}.$  If we remove the largest part $\overline{kj+r}$ from $\lambda$, then we can get an overpartition in $\mathcal{B}_{k,r}(m)$ with the largest part greater than or equal to $\overline{k(j-1)+r}$ and  not exceeding $kj+r$, and vice versa. So, \eqref{recurrence-eqn-2} is verified.

We conclude that \eqref{recurrence-eqn-1} and \eqref{recurrence-eqn-2} are valid. The proof is complete.
\end{proof}

Combining Lemmas \ref{bkrj-lem} and \ref{recurrence-lem}, we can get the generating function for the overpartitions in $\mathcal{B}_{k,r}(m)$.

\begin{theorem}\label{gen-bkrlo}
For $m\geq 1$, we have
\begin{equation}\label{use}
\sum_{\lambda\in\mathcal{B}_{k,r}(m)}x^{\ell_o(\lambda)}q^{|\lambda|}=\sum_{j\geq 0}x^jq^{m-j+k{j\choose 2}+rj}{{m-j+kj+r-1}\brack{kj+r-1}}_1.
\end{equation}
\end{theorem}

\begin{proof} Let $\lambda=(\lambda_1,\lambda_2,\ldots,\lambda_{m})$ be an overpartition in $\mathcal{B}_{k,r}(m)$. We consider the number of overlined parts in $\lambda$ based on the largest part of $\lambda$. There are two cases.

 Case 1: $1\leq \lambda_1\leq r$. In this case, there are no overlined parts in $\lambda$, and so $\ell_o(\lambda)=0$.

   Case 2: $\overline{k(j-1)+r}\leq \lambda_1\leq kj+r$, where $j\geq 1$.  In this case, there are $j$ overlined parts in $\lambda$, which are $\overline{k(j-1)+r}$,\ldots, $\overline{k+r}$, $\overline{r}$. So, we have $\ell_o(\lambda)=j$.

   Combining with Lemmas \ref{bkrj-lem} and \ref{recurrence-lem}, we get
   \begin{align*}
   &\quad\sum_{\lambda\in\mathcal{B}_{k,r}(m)}x^{\ell_o(\lambda)}q^{|\lambda|}\\
   &=g_{k,r}(m,1)+g_{k,r}(m,2)+\cdots+g_{k,r}(m,r)\\
   &\quad+\sum_{j\geq1}x^j\left\{g_{k,r}(m,\overline{k(j-1)+r})+g_{k,r}(m,{k(j-1)+r+1})+\cdots+g_{k,r}(m,{kj+r})\right\}\\
   &=q^{-r}g_{k,r}(m+1,\overline{r})+\sum_{j\geq1}x^jq^{-kj-r}g_{k,r}(m,\overline{kj+r})\\
   &=\sum_{j\geq0}x^jq^{-kj-r}g_{k,r}(m+1,\overline{kj+r})\\
   &=\sum_{j\geq0}x^jq^{-kj-r}q^{(m+1)-(j+1)+k{{j+1}\choose 2}+r(j+1)}{{(m+1)-(j+1)+kj+r-1}\brack{kj+r-1}}_1\\
    &=\sum_{j\geq0}x^jq^{m-j+k{{j}\choose 2}+rj}{{m-j+kj+r-1}\brack{kj+r-1}}_1.
   \end{align*}

   This completes the proof.
   \end{proof}

  Now, we are in a position to give a proof of \eqref{gen-m-kr-eqn}.
   \begin{proof}[Proof of \eqref{gen-m-kr-eqn}] By Theorem \ref{m-b-class}, we know that $\mathcal{M}_{k,r}$ is a separable overpartition class with modulus $k$ and $\mathcal{B}_{k,r}$ is the basis of $\mathcal{M}_{k,r}$. Note that $\mathcal{B}_{k,r}(m)$ is the set of overpartitions in $\mathcal{B}_{k,r}$ with  $m$ parts for $m\geq 1$, then by \eqref{useful}, we have

  \begin{equation}\label{M-final}
  \sum_{\pi\in\mathcal{M}_{k,r}}x^{\ell_o(\pi)}z^{\ell(\pi)}q^{|\pi|}=1+\sum_{m\geq 1}\frac{z^m}{(q^k;q^k)_m}\sum_{\lambda\in\mathcal{B}_{k,r}(m)}x^{\ell_o(\lambda)}q^{|\lambda|}.
  \end{equation}

   Substituting \eqref{use} into \eqref{M-final}, we get
  \begin{align*}
   &\quad\sum_{\pi\in\mathcal{M}_{k,r}}x^{\ell_o(\pi)}z^{\ell(\pi)}q^{|\pi|}\\
   &=1+\sum_{m\geq 1}\frac{z^m}{(q^k;q^k)_m}\sum_{j\geq0}x^jq^{m-j+k{{j}\choose 2}+rj}{{m-j+kj+r-1}\brack{kj+r-1}}_1\\
   &=\sum_{m\geq 0}\frac{z^m}{(q^k;q^k)_m}\sum_{j=0}^mx^jq^{m-j+k{{j}\choose 2}+rj}{{m-j+kj+r-1}\brack{kj+r-1}}_1\\
   &=\sum_{j\geq 0}\sum_{m\geq j}\frac{z^m}{(q^k;q^k)_m}x^jq^{m-j+k{{j}\choose 2}+rj}{{m-j+kj+r-1}\brack{kj+r-1}}_1\\
   &=\sum_{j\geq 0}\sum_{n\geq 0}\frac{z^{n+j}}{(q^k;q^k)_{n+j}}x^jq^{n+k{{j}\choose 2}+rj}{{n+kj+r-1}\brack{kj+r-1}}_1.
  \end{align*}
  So, \eqref{gen-m-kr-eqn} is verified. The proof is complete.
  \end{proof}


\bibliographystyle{amsplain}

\end{document}